\documentclass
{amsart}
\usepackage{graphicx}
\newtheorem{thm}{Theorem}[section]
\theoremstyle{definition}
\newtheorem{cor}[thm]{Corollary}
\newtheorem{prop}[thm]{Proposition}
\newtheorem{defn}[thm]{Definition}

\newtheorem{lem}[thm]{Lemma}

\newtheorem{rem}[thm]{Remark}
\newtheorem{ex}[thm]{Example}

\numberwithin{equation}{section}
 \usepackage[paperwidth=165mm, paperheight=235mm, twoside, hmargin={25mm,20mm}, vmargin={20mm,20mm} ]{geometry}
\begin{document}
\title[Second semimodules over commutative semirings]
{Second semimodules over commutative semirings}

\author[Faranak Farshadifar]%
{Faranak Farshadifar}

\newcommand{\acr}{\newline\indent}
\address{Department of Mathematics Education, Farhangian University, P.O. Box 14665-889, Tehran, Iran.}
\email{f.farshadifar@cfu.ac.ir}

\subjclass[2010]{13C13, 16Y60}%
\keywords {Semiring, semimodule, second subsemimodule, maximal second subsemimodule}

\begin{abstract}
Let $R$ be a semiring.  We say that a non-zero subsemimodule $S$ of an $R$-semimodule $M$ is
 \emph{second} if for each $a \in R$, we have $aS=S$ or $aS=0$.
The aim of this paper is to study the notion of second subsemimodules of
semimodules over commutative semirings.
\end{abstract}
\maketitle
\section{Introduction}
\noindent
A \textit{semiring} is a non-empty set $R$ together with two binary
operations addition $(+)$ and multiplication $(\cdot)$ such that $(R,+)$ is a commutative monoid with identity element $0$; $(R,.)$ is a monoid with identity element $1\neq 0$;  $0a=0=a0$ for all $a\in R$; $a(b+c)=ab+ac$ and $(b+c)a=ba+ca$ for every $a,b,c\in R$. $R$ is  called a \textit{commutative semiring} if the monoid $(R,.)$ is commutative. In this paper we assume that all semirings are commutative.

A non-empty subset $I$ of a semiring $R$ is called an \textit{ideal} of $R$ if $a+b\in I$ and $ra\in I$ for all $a,b\in I$ and $r\in R$.  An ideal $I$ of a semiring $R$ is \textit{subtractive} if $a+b\in I$ and $b\in I$ imply that
$a\in I$ for all $a,b\in R$.
Let $(M,+)$ be an additive abelian monoid
with additive identity $0_{M}$. Then $M$ is called an \textit{$R$-semimodule} if there exists a scalar multiplication $R\times M\rightarrow M$ denoted by $(r,m)\mapsto rm$, such that $(rs^{\prime})m=r(r^{\prime}m)$; $r(m+m^{\prime})= rm+rm^{\prime}$; $(r+r^{\prime})m=rm+r^{\prime}m$; $1m=m$ and $r0_{M}=0_{M}=0m$ for all $r,r^{\prime}\in R$ and all $m,m^{\prime}\in M$.
A \textit{subsemimodule} $N$ of a semimodule $M$ is a non-empty subset of $M$ such that $m+n\in N$ and $rn\in N$ for all $m,n\in N$ and $r\in R$.
A subsemimodule $N$ of an $R$-semimodule $M$ is called a \textit{subtractive subsemimodule}
or a \textit{$k$-subsemimodule} if $x, x+y \in N$ implies $y \in N$. It is clear that $N$ is subtractive
if and only $N =\overline{N}$, where $\overline{N} =\{r \in R | r + x = y \ for \ some \ x, y \in N \}$. For each subsemimodule $N$ of $M$, $\overline{N}$ is a $k$-subsemimodule of $M$. $M$ itself is a $k$-subsemimodule of $M$ and $0$ is also a
$k$-subsemimodule of $M$.

A proper subsemimodule $N$ of an $R$-semimodule $M$ is said to be \textit{prime} in $M$,
if $rx \in N$ with $r \in R$ and $x \in M$ implies $r \in (N:_RM)$ or $x \in N$ \cite{MR2397473}.
The concept of second submodule of an $R$-module (as a dual notion of prime submodules)  was introduced and studied by S.Yassemi in 2001. A non-zero submodule $N$ of an $R$-module $M$ is called  \emph{second} if for each $a \in R$, we have $aN=N$ or $aN=0$ \cite{Y01}.
This notion has obtained a great attention by many authors and now there is a considerable amount of research concerning this class of modules. For more information about this class of modules we refer the reader to \cite{FF2028}.
The algebraic structure of semirings, that are considered as a generalization of rings, plays an important role in different
branches of mathematics, especially in applied sciences and computer engineering.
The purpose of this paper is to study second semimodules as a dual notion of prime semimodules and extend some of the results of \cite{AF11, AF12, Y01} to semimodules  over commutative semirings.

\section{Second subsemimodules}
We begin with the following definition.
\begin{defn}\label{2.1}
We say that a non-zero subsemimodule $S$ of an $R$-semimodule $M$ is
 \emph{second} if for each $a \in R$, we have $aS=S$ or $aS=0$.
Also, a non-zero $R$-semimodule $M$ is said to be \textit{second}, if $M$ is a second subsemimodule of $M$.
\end{defn}

\begin{prop}\label{27.6}
Let $N$ be a subsemimodule of an $R$-semimodule $M$. Then the following are equivalent:
\begin{itemize}
\item [(a)] $N$ is a second subsemimodule of $M$;
\item [(b)] $N\not=0$ and $aS \subseteq K$ implies that $aS=0$ or $S\subseteq K$ for each $a \in R$ and subsemimodule $K$ of $M$.
\end{itemize}
\end{prop}
\begin{proof}
$(a)\Rightarrow (b)$.
This is clear.

$(b)\Rightarrow (a)$.
This follows from the fact that $aN\subseteq aN$.
\end{proof}

\begin{defn}
Let $M$ be an $R$-semimodule. We say that a subsemimodule $N$ of $M$ is
a \textit{minimal subsemimodule} of $M$ if there is no subsemimodule $K$ of $M$ satisfying $0\subset K \subset N$.
\end{defn}

\begin{ex}
Assume that $\Bbb Z_0^+$ is the set of non-negative integers and consider the $\Bbb Z_0^+$-semimodule $M=\Bbb Z_{16}$ and take $N=\{\bar{0}, \bar {8}\}$ as a subsemimodule of $M$. Then $N$ is a minimal subsemimodule of $M$.
\end{ex}

\begin{rem}
Clearly, every minimal subsemimodule of $R$-semimodule $M$ is a second subsemimodule of $M$. But the converse is not true in general. For example, consider the $\Bbb Z_0^+$-semimodule $M=\Bbb Q^+$. Then $M$ is a second subsemimodule of $M$ which is not a minimal subsemimodule of $M$.
\end{rem}

An $R$-semimodule $M$ is called a \textit{comultiplication $R$-semimodule} if for any subsemimodule
$N$ of $M$ there exists an ideal $I$ of $R$ such that $N = (0 :_M I)$ \cite{MR4728970}. Clearly, $(0 :_M I)$ is a subtractive subsemimodule of $M$ for each ideal $I$ of $R$. Thus if $M$ is a comultiplication $R$-semimodule, then every subsemimodule of $M$ is a subtractive subsemimodule.

 An $R$-semimodule $M$ is called a $k$-comultiplication $R$-semimodule
if for any subtractive subsemimodule $N$ of $M$, there exists an ideal $I$ of $R$ such that $N = (0 :_
M I)$ \cite{MR4728970}.
\begin{prop}\label{2.2}
Let $R$ be a semiring, $M$ an $R$-semimodule and $N$ a subsemimodule of $M$. Then we have the following.
\begin{itemize}
\item [(a)] If $N$ is a second subsemimodule of $M$, then $Ann_R(N)$ is a prime $k$-ideal of $R$. In this case, we say that $N$ is
$Ann_R(N)$-second.
\item [(b)] If $N$ is a subsemimodule of a comultiplication $R$-semimodule $M$ such that $Ann_R(N)$ is prime ideal of $R$, then
$N$ is a second subsemimodule of $M$.
\end{itemize}
\end{prop}
\begin{proof}
(a) First note that by \cite[Proposition 2.3]{MR4835184}, $Ann_R(N)$ is a $k$-ideal of $R$.
Let $N$ be a second subsemimodule of $M$. Then $N\not=0$ and so  $Ann_R(N)\not=R$.
Now suppose that $rs \in Ann_R(N)$. Then $sN\subseteq (0:_Mr)$. Thus $sN=0$ or $rN=0$, as needed.

(b) Suppose that $N$ is a subsemimodule of a comultiplication $R$-semimodule $M$ such that $Ann_R(N)$ is prime ideal of $R$. Let $aN\subseteq K$ for some $a \in R$ and subsemimodule $K$ of $M$. As $M$ is a comultiplication $R$-semimodule, there exists an ideal $I$ of $R$ such that $K=(0:_MI)$. Therefore, $aI \subseteq Ann_R(N)$.
If $I \subseteq Ann_R(N)$, then $N\subseteq K$ and we are done. So suppose that $b \in I \setminus Ann_R(N)$. Then $ab \in Ann_R(N)$. This implies that $a \in  Ann_R(N)$ or $b \in  Ann_R(N)$. Since $b \not \in Ann_R(N)$, we have $a \in  Ann_R(N)$, as needed.
\end{proof}

The following example shows that the converse of Proposition \ref{2.2} (a) is not true in general.
\begin{ex}
Let $R$ be  $\Bbb Z^* = \Bbb Z^+ \cup \{0\}$. Then $M = \Bbb Z^* \times \Bbb Z^*$ is an $R$-semimodule.
Consider the subsemimodule $N = 0 \times 4\Bbb Z^*$ of $M$.  Then $Ann_R(N)=0$ is prime ideal of $R$ but $N$ is not a second subsemimodule of $M$. Because $2N\not=N$ and $2N\not=0$.
\end{ex}

A proper ideal $I$ of a semiring $R$ is said to be a \textit{strong ideal} if for
each $a\in I$ there exists $b \in I$ such that $a+b = 0$ \cite{MR2652253}.
\begin{prop}\label{t3.2}
Let $M$ be a finitely generated comultiplication $R$-semimodule and $P$ be a strong prime $k$-ideal of $R$ containing $Ann_R(M)$. Then
$(0:_MP)$ is a second subsemimodule of $M$.
\end{prop}
\begin{proof}
By Proposition \ref{2.2} (b), it is enough to show that $Ann_R((0:_MP))=P$. Let
$r(0:_MP)=0$. Then $(0:_MP)\subseteq (0:_Mr)$. Since $M$ is a
comultiplication semimodule, we have
$$
M=(PM:_MP)\subseteq ((0:_MAnn_R(PM)):_MP)=((0:_MP):_MAnn_R(PM))
$$
$$
\subseteq ((0:_Mr):_MAnn_R(PM))=(PM:_Mr).
$$
It follows that $rM \subseteq PM$.
This implies that $r \in P$ by \cite[Proposition 2.3]{MR2652253}.
Therefore, $P=Ann_R((0:_MP))$
because the reverse inclusion is clear.
\end{proof}

\begin{prop}\label{28.51}
Let $S$ be a second subsemimodule of a comultiplication $R$-semimodule $M$ and let $N_1, \ldots, N_t$ be subsemimodules of $M$. Then the following statements are equivalent:
\begin{itemize}
\item [(a)] $S\subseteq N_j$ for some $j$ with $1 \leq j \leq n$;
\item [(b)] $S\subseteq \sum ^t_{i=1}N_i$.
\end{itemize}
\end{prop}
\begin{proof}
$(a)\Rightarrow (b)$. This is clear.

$(b)\Rightarrow (a)$.
As $M$ is a comultiplication $R$-semimodule, we have $N_j=(0:_MAnn_R(N_j))$ for all $j$ with $1 \leq j \leq n$.
Thus $S\subseteq \sum ^t_{i=1}N_i $ implies that
$$
 S\subseteq \sum ^t_{i=1}(0:_MAnn_R(N_i))=(0:_M\cap_{i \in I}Ann_R(N_j))
$$
by \cite[Lemma 4.4.]{MR4728970}.
Hence, $\cap_{i \in I}Ann_R(N_j)\subseteq Ann_R(S)$. Now by using Proposition \ref{2.2} (a) and \cite[Lemma 2.4]{MR2732626},
$Ann_R(N_j)\subseteq Ann_R(S)$ for some $j$ with $1 \leq j \leq n$. This implies that $(0:_MAnn_R(S))=S\subseteq N_j=(0:_MAnn_R(N_j))$.
\end{proof}

Let $M$, $N$ be $R$-semimodules, and $f$ be a map from $M$ to $N$. $f$ is said to be a \textit{semimodule
homomorphism} (see \cite{MR1746739}) if
\begin{itemize}
\item [(1)]  $f (x + y) = f (x) + f (y)$ for all $x, y \in  M$;
\item [(2)]  $f (rx) = rf (x)$ for all $r \in R$, $x \in  M$.
\end{itemize}
$Ker(f ) :=\{a \in M| f (a) = 0\}$ is called the \textit{kernel of $f$}. Also, $f(M) := \{f (a) | a \in M\}$.
It is easy to see that $Ker (f)$ is a subsemimodule of $M$ and $f(M)$ is a subsemimodule of $N$.
A semimodule homomorphism $f:M\rightarrow N$ is said to be
 \textit{$k$-regular} (\textit{kernel-regular}) if $f(x_1) = f(x_2)$, then $x_1+k_1 = x_2 +k_2$ for some $k_1, k_2\in Ker(f)$ \cite{MR2536767}.

\begin{prop}\label{2.7}
Let $f:  M \rightarrow M^{\prime}$ be a homomorphism of $R$-semimodules with $Ker (f)=0$, where $M$ is a subtractive semimodule. Then we have the following.
\begin{itemize}
\item [(a)] If $S$ is a second subsemimodule of $M$ and $f$ is $k$-regular, then $f(S)$ is a second subsemimodule of $M^{\prime}$.
\item [(b)] If $S^{\prime}$ is a second subsemimodule of $f(M)$, then $f^{-1}(S^{\prime})$ is a second subsemimodule of $M$.
\end{itemize}
\end{prop}
\begin{proof}
(a) Let $S$ be a second subsemimodule of $M$.
If $f(S)=0$, then $f^{-1}(f(S))=S+Ker (f)=S=0$ by \cite[Proposition 3.2 (ii)]{MR2536767} because $f$ is $k$-regular and $M$ is subtractive. This contradiction shows that $f(S)\not=0$. Now assume that $af(S) \not =f(S)$ and $af(S) \not =0$ for some $a\in R$. Then we have $aS \not =S$ and  $aS \not =0$. These are contradictions.

(b) Let $S^{\prime}$ be a second subsemimodule of $f(M)$. If $f^{-1}(S^{\prime})=0$, then $S^{\prime}=S^{\prime}\cap f(M)=f(f^{-1}(S^\prime))=0$. This contradiction shows that  $f^{-1}(S^{\prime})\not=0$. Now, let $af^{-1}(S^{\prime})\not = f^{-1}(S^{\prime})$ and $af^{-1}(S^{\prime})\not =0$ for some $a\in R$. Then $aS^{\prime}\not =S^{\prime}$  and $aS^{\prime}\not =0$.  These are contradictions.
\end{proof}

\noindent
The following corollary is now evident.
\begin{cor}\label{2.8}
If $N$ and $K$ are subsemimodules of an $R$-semimodule $M$ with $N$ is subtractive, $K\subseteq N\subseteq M$, and $K$ is a second subsemimodule of $N$, then $K$ is a a second subsemimodule of $M$.
\end{cor}

\begin{prop}\label{df2.1}
Let $S$ be a subsemimodule of an $R$-semimodule $M$ such that $Ann_R(S)$ is a maximal ideal of $R$. Then $S$
is a second subsemimodule.
\end{prop}
\begin{proof}
As $Ann_R(S)\not=0$, we have $S\not=0$. Let $r \in R$. If $rS\not =0$, then $Ann_R(S) \subseteq Ann_R(S) +Rr\subseteq R$ implies that
$Ann_R(S) +Rr= R$. Thus $Ann_R(S)S +rS= S$. Hence $rS=S$, as needed.
\end{proof}

\begin{prop}\label{df2.9}
Let $\mathfrak{p} \in Spec(R)$. Then the following hold:
\begin{itemize}
\item [(a)] The sum of $\mathfrak{p}$-second semimodules is a $\mathfrak{p}$-second semimodule.
\item [(b)] Every product of $\mathfrak{p}$-second semimodule is a $\mathfrak{p}$-second semimodule.
\item [(c)] Every non-zero quotient of a $\mathfrak{p}$-second semimodule is likewise $\mathfrak{p}$-second.
\end{itemize}
\end{prop}
\begin{proof}
This is straightforward.
\end{proof}

\begin{prop}\label{8l3.14}
Let $M$ be an $R$-semimodule. If every non-zero subsemimodule of $M$
is second, then for each subsemimodule $K$ of $M$ and each ideal $I$ of $R$, we have
$(K :_M I) = (K :_M I^2)$. Also for any two ideals $A,B$ of $R$, $(K :_M A)$ and
$(K :_M B)$ are comparable.
\end{prop}
\begin{proof}
This is straightforward.
\end{proof}

\begin{defn}
We say that a subsemimodule $N$ of an $R$-semimodule $M$
is \emph{coidempotent} if $N=(0:_MAnn_R^2(N))$.
Also, an $R$-semimodule $M$ is said to be \emph{fully coidempotent}
if every subsemimodule of $M$ is coidempotent.
\end{defn}

 \begin{prop}\label{t2.5}
Let $M$ be a fully coidempotent $R$-semimodule.
\begin{itemize}
  \item [(a)] $M$ is a comultiplication semimodule.
  \item [(b)] Every subsemimodule and every homomorphic image of $M$ is fully coidempotent.
  \item [(c)] $M$ is Hopfian.
\end{itemize}
\end{prop}
 \begin{proof}
(a) This is clear.

(b) It is easy to see that every subsemimodule of $M$ is
fully coidempotent. Now let $N$ be a
submodule of $M$ and $K/N$ be a subsemimodule of $M/N$. By part (a),
$M$ is a comultiplication $R$-semimodule. Hence $K=(0:_MAnn^2_R(K))$
implies that $K=(0:_MAnn^3(K))$. Thus
$$
(0:_{M/N}Ann^2(K/N))=(0:_MAnn_R(N)Ann^2_R(K/N))/N
$$
$$
\subseteq (0:_MAnn^3_R(K))/N=K/N.
$$
Therefore, $K/N =(0:_{M/N}Ann^2_R(K/N))$.

(c) Let $f:M \rightarrow M$ be
an epimorphism. Then by assumption and part (a),
$Ker(f)=(0:_MI)=(0:_MI^2)$, where $I=Ann_R(ker(f))$. If $y \in
Ker(f)$, then $y \in (0:_{f(M)}I)$ because $f$ is epic. Thus
$y=f(x)$ for some $x \in M$ and $f(x)I=0$. Hence $xI^2=0$. It
follows that $xI=0$. Therefore, $y =0$, as required.
\end{proof}

\begin{thm}\label{8lfff3.14}
Let $M$ be a fully coidempotent $R$-semimodule. Then every second subsemimodule of $M$ is a minimal subsemimodule of $M$.
\end{thm}
\begin{proof}
 Let $S$ be a second subsemimodule of $M$ and $K$ be a
subsemimodule of $S$. If $Ann_R(K) \subseteq Ann_R(S)$, then $S
\subseteq K$ because $M$ is a comultiplication $R$-semimodule
by Proposition \ref{t2.5}
(a). If $Ann_R(K) \not \subseteq
Ann_R(S)$, then there exists $r \in Ann_R(K)-Ann_R(S)$. Since $S$ is
second $rS=S$. By Proposition \ref{t2.5} (b), $S$ is
fully coidempotent. Hence by Proposition \ref{t2.5} (c), $S$ is Hopfian.
It follows that the epimorphism $r: S
\rightarrow S$ is an isomorphism. Hence $rK=0$ implies that $K=0$, as
required.
\end{proof}

\begin{lem}\label{2.98}
Let $R = R_1 \times R_2$, where $R_j$ is a commutative semiring for all
$j \in \{1, 2\}$ and $J_1$ is an ideal of $R_1$. If $J_1$ is a prime ideal of $R_1$, then  $J_1 \times  R_2$ is a prime ideal of $R$.
\end{lem}
\begin{proof}
Let $J_1$ be a prime ideal of $R_1$ and $(a, b)(c,d) \in  J_1 \times R_2$. Then $ac \in  J_1$. Since $J_1$ is a prime ideal of $R_1$, $a \in J_1$ or $c \in J_1$. Thus $(a, b) \in J_1\times R_2$ or $(c,d)\in J_1 \times  R_2$. Hence,  $J_1\times R_2$ is a prime ideal of $R$. The proof of the converse is trivial.
\end{proof}

\begin{thm}\label{2.998}
 Let $R = R_1 \times R_2$, where $R_j$ is a commutative semiring for all
$j \in \{1, 2\}$ and $J_1$ is an ideal of $R_1$ and $J_2$ is an ideal of $R_2$ such that $J = J_1 \times  J_2$ is
an ideal of $R$. Then the following statements are equivalent:
\begin{itemize}
\item [(a)] $J$ is a prime ideal of $R$;
\item [(b)] $J=J_1 \times R_2$ for some prime ideal $J_1$ of $R_1$ or  $J=R_1 \times J_2$ for some prime ideal $J_2$ of $R_2$.
\end{itemize}
\end{thm}
\begin{proof}
$(a) \Rightarrow (b)$
Let $J=J_1 \times J_2$ be a prime ideal of $R$, where $J_1$ is an ideal of $R_1$ and $J_2$ is
an ideal of $R_2$. Then $(0, 1)(1,0) \in  J_1 \times J_2=J$ implies that $(1,0) \in  J_1 \times J_2$ or $(0, 1) \in  J_1 \times J_2$. Therefore,  $J_1=R$ or $J_2=R$. Thus $J=R_1 \times J_2$ or $J=J_1 \times R_2$.
Assume that $J = J_1 \times R_2$ for some proper ideal $J_1$ of $R_1$. Now, we
show that $J_1$ is a prime ideal.  Assume contrary that $J_1$ is not a prime ideal
of $R_1$. Then there exist elements $x, y \in R_1$ such that $xy \in J_1$ but neither $x \in J_1$
nor $y \in J_1$. Thus, $(x, 1_{R_2})(y, 1_{R_2}) \in J_1 \times  R_2$ implies that $(x, 1_{R_2}) \in J_1 \times R_2$ or $(y, 1_{R_2}) \in J_1 \times R_2 $. Consequently, $x \in J_1$  or $y \in J_1$, which gives a contradiction. Hence, $J_1$ is a prime ideal of $R_1$.

$(a) \Rightarrow (b)$
This follows from Lemm \ref{2.98}.
\end{proof}

\begin{lem}\label{2.9}
Let $R_i$ be a commutative semiring with identity $1_{R_i}$ and $M_i$ be a faithful $R_i$-semimodule, for $i = 1, 2$.
Let $R = R_1\times R_2$, $M = M_1 \times M_2$. Suppose that $S_i$ is a subsemimodule of $M_i$ for
$i = 1, 2$ such that $S = S_1\times S_2$ is a subsemimodule of $M$. Then the followings are equivalent:
\begin{itemize}
\item [(a)] $S$ is a second subsemimodule of $M$;
\item [(b)] $S_1$ is a second subsemimodule of $M_1$ and $S_2 = 0$ or $S_1 = 0$ and $S_2$ is a second
subsemimodule of $M_2$.
\end{itemize}
\end{lem}
\begin{proof}
$(a)\Rightarrow (b)$.
Clearly, $M$ is a faithful $R$-module. By Proposition \ref{2.2}, $Ann_R(S)$ is a prime ideal of $R$. Thus we have either $Ann_{R_1}(S_1) = R_1$ or $Ann_{R_2}(S_2)= R_2$ by Theorem \ref{2.998}. So we can assume that $Ann_{R_1}(S_1) = R_1$. Then $S_1=0$. Now we prove that $S_2$ is a second subsemimodule of $M_2$. To see this, let $a_2 \in R_2$. Then by assumption $(0, a_2)S=0$ or $(0, a_2)S=S$.
Thus we have $a_2S_2=0$ or $a_2S_2=S_2$, as needed.

$(b)\Rightarrow (a)$.
Assume that $S_2= 0$ and $S_1$ is a second subsemimodule of $M_1$. We show that $S$ is a second subsemimodule of $M$. So let $(a_1,a_2) \in  R_1\times R_2$. Then $a_1S_1=0$ or $a_1S_1=S_1$. Thus
$(a_1, a_2)(S_1\times 0)=0$ or $(a_1, a_2)(S_1\times 0)=S_1\times 0$.
Hence, $S$ is a second subsemimodule of $M$.
\end{proof}

\begin{thm}\label{2.10}
Let $R_i$ be a commutative semiring with identity $1_{R_i}$ and $M_i$ be a faithful $R_i$-semimodule, for $i = 1, 2, \ldots, n$ where $n \geq 2$.
Let $R=R_1\times R_2\times\cdots \times R_n$, $M=M_1\times M_2\times \cdots \times M_n$, and $S = S_1\times S_2\times\cdots \times S_n$,
where $S_i$ is a subsemimodule of $M_i$, $1 \leq i \leq n$. Then the followings are equivalent:
\begin{itemize}
\item [(a)] $S$ is a second subsemimodule of $M$;
\item [(b)] $S_j$ is a second subsemimodule of $M_j$ for some $j \in \{1,2,\ldots ,n\}$ and $S_i = 0$ for
each $i \not= j$.
\end{itemize}
\end{thm}
\begin{proof}
We use induction on $n$. By Lemma \ref{2.9}, the claim is true if $n = 2$. So, suppose that the claim is true for each $k  \leq n-1$ and let $k = n$. Put $Q= S_1\times S_2\times\cdots \times S_{n-1}$, $\acute{R} = R_1\times R_2\times \cdots \times R_{n-1}$,
and $\acute{M} = M_1\times M_2\times\cdots \times M_{n-1}$, by Lemma \ref{2.9}, $S = Q\times S_n$ is a second subsemimodule of
$M = \acute{M} \times M_n$ if and only if $Q$ is a second subsemimodule of $\acute{M}$ and $S_n = 0$ or $Q=0$ and
$S_n$ is a second subsemimodule of $M_n$. Now the rest follows from induction hypothesis.
\end{proof}

An $R$-semimodule $M$ is said to be \textit{Noetherian} if $M$ satisfies the $ACC$ on its $R$-subsemimodules. Also, $M$ is said to be \textit{Artinian} if $M$ satisfies the $DCC$ on its $R$-subsemimodules \cite{Abu07}.

An $R$-semimodule $M$ is said to be \textit{simple} if it has no proper subsemimodules. Also,  $M$ is said to be \textit{semisimple} if it
is a direct sum of its simple subsemimodules \cite{MR2536767}.

\begin{prop}\label{l2.9}
Let $M$ be a finitely generated second $R$-semimodule. Then $Ann_R(M)$ is a maximal ideal of $R$.
\end{prop}
\begin{proof}
Since $M$ is second, $M \not =0$. There exists a maximal submodule $U$ of $M$  because $M$ is finitely generated by \cite[Proposition 2.1]{MR2536767}.  Clearly, $Ann_R(M)\subseteq Ann_R(M/U)$.  If $Ann_R(M/U)\not \subseteq Ann_R(M)$, then there is $r\in R$ such that $rM\subseteq U$ and $rM\not=0$. This implies that $rM=M$ since $M$ is second. Therefore, $U=M$ which is a contradiction since $U$ is maximal. Thus  $Ann_R(M)=Ann_R(M/U)$. Now, as $M/U$ is simple, $Ann_R(M)=Ann_R(M/U)$ is a maximal ideal of $R$. 
\end{proof}

\begin{defn}\label{d3.5}
 Let $N$ be a subsemimodule of an $R$-semimodule $M$. We define the \emph{socle} of $N$ as the sum of all second subsemimodules of $M$ contained in $N$ and it is denoted by $sec(N)$.
In case $N$ does not contain any second subsemimodule, the socle of $N$ is defined to be \textbf{$(0)$}. Also, we say that $N \not =0$ is a \emph{socle subsemimodule of $M$} if $sec(N)=N$.
\end{defn}

\begin{thm} \label{t3.8}
Let $M$ be a Noetherian $R$-semimodule.
Then there is a second subsemimodule of $M$ which contains each socle subsemimodule of $M$.
\end{thm}
\begin{proof}
If $M$ does not contain any second subsemimodules, then the
results is true vacuously. So we assume that $M$ contains a second
subsemimodule. Let $\sum$ be the set of all subsemimodules of $M$ which
can be expressed as a sum of a finite number of second subsemimodules.
Since $M$ is Noetherian, $\sum$ has a maximal member, $S$ say.
Hence, there exist second subsemimodules $S_1, S_2, ...,S_n$ of $M$
such that $S=S_1+S_2+...+S_n$. Let $L$ be any second subsemimodule of
$M$. Then
$$
S=S_1+S_2+...+S_n \subseteq L+S_1+S_2+...+S_n \in \sum.
$$
By the maximality of $S$, we have $S=L+S_1+S_2+...+S_n$. Hence, $L \subseteq S$. Thus,  $S$ contains each socle subsemimodule of $M$. 
\end{proof}

\begin{defn}
Let $M$ be an $R$-semimodule. We say that a second subsemimodule $N$ of $M$ is a \emph{maximal second subsemimodule} of a subsemimodule $K$ of $M$, if $N \subseteq K$ and there does not exist a second
subsemimodule $L$ of $M$ such that $N \subset L \subset K$.
\end{defn}

\begin{thm}\label{t3.6}
Let $M$ be an $R$-semimodule. If $M$ satisfies
the descending chain condition on socle subsemimodules, then every
non-zero subsemimodule of $M$ has only a finite number of maximal
second subsemimodules.
\end{thm}
\begin{proof}
Suppose that there exists a non-zero subsemimodule $N$ of $M$
such that $N$ has an infinite number of maximal second subsemimodules
and look for a contradiction. Then $sec(N)$ is a socle
subsemimodule of $M$ and it has an infinite number of maximal second
subsemimodules. Let $S$ be a socle subsemimodule of $M$ chosen minimal
such that $S$ has an infinite number of maximal second
subsemimodules. Then $S$ is not second. Thus there exists a subsemimodule
$L$ of $M$ and an ideal $I$ of $R$ such that $L\subset S$ and $S
\not \subseteq (0:_MI)$. Let $V$ be a maximal second subsemimodule of
$M$ contained in $S$. Then $V \subseteq (0:_SI)$ or $V \subseteq
L$. By the choice of $S$, both the semimodules $(0:_SI)$ and $L$ have
only finitely many maximal second subsemimodules. Therefore, there is
only a finite number of possibilities for the semimodule $S$, which
is a contradiction.
\end{proof}

\begin{cor}\label{c3.7}
Every Artinian $R$-semimodule contains only a finite number of maximal second subsemimodules.
\end{cor}

\end{document}